\documentclass[10pt]{amsart}
\usepackage{amsmath, amsthm, amssymb, tikz-cd}
\usepackage[utf8]{inputenc}
\usepackage{rotating}
\usepackage[T1]{fontenc}
\usepackage{dsfont}
\usepackage{amscd}
\usepackage{mathabx}
\usepackage[mathscr]{euscript} 
\usepackage[all]{xy}
\usepackage{url}
\usepackage{stmaryrd}
\usepackage{comment}
\usepackage[retainorgcmds]{IEEEtrantools}
\usepackage{hyperref}

\newtheorem{theorem}{Theorem}[section]
\newtheorem{lemma}[theorem]{Lemma}
\newtheorem{fact}[theorem]{Fact}
\newtheorem{proposition}[theorem]{Proposition}

\newtheorem{corollary}[theorem]{Corollary}

\newtheorem{conjecture}[theorem]{Conjecture}
\theoremstyle{definition}
\newtheorem{definition}[theorem]{Definition}
\newtheorem{example}[theorem]{Example}
\newtheorem{remark}[theorem]{Remark}
\newtheorem{question}[theorem]{Question}
\newtheorem{notation}[theorem]{Notation}

\newcommand{\Zz}{{\mathds{Z}}}
\newcommand{\Ff}{{\mathds{F}}}
\newcommand{\Cc}{{\mathds{C}}}
\newcommand{\Rr}{{\mathds{R}}}

\newcommand{\Qq}{{\mathds{Q}}}

\newcommand{\gm}{\mathbb{G}_{\rm{m}}}

\newcommand{\cH}{\mathcal{H}}

\newcommand{\cO}{\mathcal{O}}

\newcommand{\cP}{\mathcal{P}}

\title{The class of Krasner hyperfields is not elementary}
\author[P. B{\L}ASZKIEWICZ]{Piotr B{\l}aszkiewicz$^{\clubsuit}$}
\thanks{$^{\clubsuit}$ Supported by the Narodowe Centrum Nauki grant no. 2021/43/B/ST1/00405.}
\address{$^{\spadesuit}$Instytut Matematyczny\\
Uniwersytet Wroc{\l}awski\\
Wroc{\l}aw\\
Poland}
\email{piotr.blaszkiewicz@math.uni.wroc.pl} 
\author[P. KOWALSKI]{Piotr Kowalski$^{\spadesuit}$}
\thanks{$^{\spadesuit}$ Supported by the Narodowe Centrum Nauki grant no. 2021/43/B/ST1/00405.}
\address{$^{\spadesuit}$Instytut Matematyczny\\
Uniwersytet Wroc{\l}awski\\
Wroc{\l}aw\\
Poland}
\email{pkowa@math.uni.wroc.pl} \urladdr{http://www.math.uni.wroc.pl/\textasciitilde pkowa/ }

\DeclareMathOperator{\rk}{rr}\DeclareMathOperator{\gal}{Gal}
\DeclareMathOperator{\cha}{char}
\DeclareMathOperator{\id}{id}

\begin{document}

\maketitle
\begin{abstract}
We show that the class of Krasner hyperfields is not elementary. To show this, we determine the rational rank of quotients of multiplicative groups in field extensions. Our argument uses Chebotarev's density theorem. We also discuss some related questions.
\end{abstract}

\section{Introduction}\label{intro}
The notion of a \textit{hyperfield} (\emph{hypercorps}) was introduced first by Marc Krasner in \cite{kr1} as a tool to study valued fields. In his later paper \cite{kr2}, he introduced the quotient construction of a hyperfield from a given field and a subgroup of its multiplicative group (see Theorem \ref{khyp}).

The question whether all hyperfields come from this quotient construction has been an open problem until the example of Massouros, who showed in \cite{Mass} that it is not the case. Nevertheless, the class of \textit{Krasner hyperfields} (i.e., hyperfields obtained by this quotient construction) contains a lot of known examples of hyperfields. Among them, there are the hyperfields known as RV-sorts, which are studied in the model theory of valued fields. They were first introduced by Joseph Flenner in his PhD thesis \cite{jf} as a tool to obtain (relative) quantifier elimination for valued fields. Flenner proved that RV-sorts are bi-interpretable with $amc$-structures (three sorted structures) introduced by Franz-Viktor Kuhlmann in \cite{FVK}. Currently, hyperfields in the form of RV-sorts are one of the main objects used to study model theory of valued fields (see e.g. \cite{inp, lee, burden, pierre}).

In view of the usefulness of Krasner hyperfields for the model theory of valued fields discussed above, we were motivated to study model theoretical properties of Krasner hyperfields themselves. Since the definition of Krasner hyperfields is purely algebraical, the first question we faced was: ``Is the class of Krasner hyperfields elementary?''. Based on the results of Alain Connes and Caterina Consani from \cite{cc}, we show in this paper that this class is \emph{not} elementary.

The paper is organized as follows. In Section \ref{prelim}, we collect the necessary facts and results about hyperfields as well as results from field theory.
In Section \ref{mainsec}, we prove our main algebraic result (Theorem \ref{mainalg}) and use it to show that the class of Krasner hyperfields is not elementary (Corollary \ref{main}). In Section \ref{specsec}, we quickly state and prove this algebraic result in its proper generality (Theorem \ref{gen}) and discuss some model-theoretical problems related with hyperfields and the algebraic methods used in this paper (Question \ref{q} and Conjecture \ref{c}).

We would like to thank Franz-Viktor Kuhlmann for his careful reading of this paper and his suggestions for improvement. Last but not least we would like to thank the members of the Wroc{\l}aw model theory seminar for their insightful questions and comments concerning this paper.

\section{Preliminaries}\label{prelim}
In this section, we introduce the necessary notions we are going to use throughout this paper (or we cite the necessary sources). Further we will present the results from the paper of Alain Connes and Caterina Consani \cite{cc}, where (among other things) they studied connections between Krasner hyperfields and projective geometries. Finally, we state the Chebotarev's density theorem which will play a crucial role in the arguments in Section 3.

Not everything from this section is directly needed for the arguments in Section 3, e.g. Theorem \ref{theorem}, Facts \ref{n1}, \ref{n2}, or Remark \ref{d2} will not be used directly. However, we hope that these extra results provide a greater picture and they also show how to avoid possible ``wrong paths'' in the main argument.
\subsection{Hyperfields}

The notion of a hyperfield, as one could expect, generalises the one of a field. The twist is that the addition is a multivalued operation, so instead of an element it returns a nonempty set.
\begin{definition}
A hyperfield is a tuple $(\cH, +,\cdot, 0, 1)$ where $(\cH,\cdot, 1)$ is an abelian group and $$+:\cH\times\cH\rightarrow\cP(\cH)\setminus\{\emptyset\}$$
satisfies the following axioms, where $x,y,z\in \cH$ and $+,\cdot$ are naturally extended to subsets of $\cH$:
    \begin{itemize}
            \item $x+y=y+x$ (commutativity),
            \item $(x+y)+z = x+(y+z)$ (associativity),
            \item for each $x\in \cH$, there is a unique $-x\in \cH$ such that $0\in x+(-x)$ (unique inverse),
            \item $z\in x+y\Rightarrow y\in z+(-x)$ (reversibility),
            \item $x+0=\{x\}$ (neutral element),
            \item $z\cdot (x+y)=z\cdot x+z\cdot y$ (distributivity).
    \end{itemize}
\end{definition}
Note that every field can be viewed as a hyperfield in the obvious way. For more details and preliminary notions concerning hyperfields (such as homomorphisms, hyperideals, etc.) we direct the reader to \cite{davsal, jjun, Ale1}.

We state now the theorem of Krasner which was mentioned in the introduction.
\begin{theorem}\label{khyp}
    Let $K$ be a field and $G$ a subgroup of $K^\times$. The quotient $K^\times/G$ together with an extra element $0$ and $+,\cdot$ defined as:
    \begin{itemize}
    \item $aG\cdot bG:=abG$,
    \item $aG+bG:=\{(x+y)G\mid x\in aG, y\in bG\}$
    \end{itemize}
    forms a hyperfield, where $1=G$.
\end{theorem}
\begin{notation}\label{not1}
We will abbreviate $(K^\times/G)\cup\{0\}$ from Theorem \ref{khyp} as $K/G$.
\end{notation}

\begin{definition}
    If a hyperfield $\cH$ is isomorphic to $K/G$ (as in Notation \ref{not1}), then we call it a \textbf{Krasner hyperfield}.
\end{definition}

\subsection{Projective geometries and CC-hyperfields}
The material of this subsection comes from \cite{cc}. We introduce first the definition of $CC$-hyperfields.
\begin{definition}
We call a hyperfield $(\cH, +,\cdot, 0, 1)$ a \textbf{$CC$-hyperfield}, if
\[
\forall x\in \cH\ \  x+x=\{x,0\}.
\]
\end{definition}
\begin{remark}
The name ``$CC$-hyperfield'' does not appear anywhere except this paper. In particular, it was not used in \cite{cc} where hyperfields with this property were studied. However, the name used in \cite{cc} is quite technical (``hyperfield extensions of the Krasner hyperfield $\mathbf{K}$'') and, for simplicity, we decided to refer to this class of hyperfields as $CC$-hyperfields.
\end{remark}
There is the following nice description of Krasner CC-hyperfields.
\begin{proposition}[Proposition 2.7 in \cite{cc}]
 Let $K$ be a field and $G$ be a subgroup of $K^\times$. Assume that $G\neq\{1\}$. Then the hyperfield $K/G$ is a $CC$-hyperfield if and only if
$\{0\}\cup G$ is a subfield of $K$.
\end{proposition}
The class of CC-hyperfields is also closely related to \emph{projective geometries} (in the sense of \emph{incidence geometry}, for the necessary notions see e.g. \cite{incidence}).
\begin{proposition}[\cite{cc} Proposition 3.1]\label{definable}
If $\cH$ is a CC-hyperfield, then there is a unique projective geometry on $\cH\setminus \{0\}$ such that for distinct $x,y\in \cH\setminus \{0\}$, the unique line through $x$ and $y$ coincides with $\{x,y\}\cup x+y$.
\end{proposition}
\begin{notation}
We will denote the above projective geometry by $\cP_\cH$.
\end{notation}
\begin{remark}\label{desa}
    If $\cH$ is a Krasner CC-hyperfield, then the corresponding projective geometry $\cP_\cH$ is the classical one, which we will see below. By Proposition \ref{definable}, $\cH=L/K^\times$, where $K$ is a subfield of $L$. Then we can view $L$ as a vector space over $K$ and consider the classical projective geometry associated to this vector space. This geometry happens to be exactly the projective geometry associated to the CC-hyperfield $L/K^\times$. In particular, such a projective geometry is always \emph{Desarguesian} and we have
    $$\dim\left(\cP_\cH\right)+1=[L:K].$$
\end{remark}
We finish this subsection with a result from \cite{cc} which will tell us later that we need to focus on Krasner CC-hyperfields of dimension one, where by the dimension of a CC-hyperfield, we always mean the dimension of its associated projective geometry.
\begin{theorem}[\cite{cc} Theorem 3.8]\label{theorem}
Let $\cH$ be a $CC$ hyperfield. Assume that the projective geometry $\cP_\cH$ is Desarguesian and of dimension at least 2. Then there exists a unique pair $(L,K)$ where $L$ is a field, and $K$ is its subfield such that
\[
\cH=L/K^\times.
\]
\end{theorem}

\subsection{Model theory}
In this subsection, we specify the model-theoretical set-up which is needed to work with hyperfields. We also show several reduction results.

We start with specifying the first-order language of hyperfields.
\begin{definition}\label{lang}
Let us set the language of hyperfields as the tuple $(\oplus,\ominus,\odot,^{-1},\underline{0},\underline{1})$, where:
\begin{itemize}
  \item $\odot$ is a binary function symbol interpreted as a multiplication,
  \item $^{-1}$ is unary function symbol interpreted as a multiplicative inverse,
  \item   $\oplus$ is a ternary relation symbol encoding the hyperaddition (so, in a hyperfield we will have: $\oplus(x,y,z)$ if and only if $z\in x+y$),
  \item  $\ominus$ is a unary function symbol encoding the additive inverse (so, $\oplus(x,\ominus x,0)$ holds in a hyperfield),
  \item $\underline{1}$ and $\underline{0}$ are constant symbols corresponding to the neutral elements of the multiplication and the hyperaddition respectively.
\end{itemize}
\end{definition}
Clearly, the class of hyperfields can be first-order axiomatized in the language above. Let us state the following well-known result.
\begin{fact}\label{n1}
The class $\mathcal{C}$ of structures (in a fixed language) is elementary if and only if $\mathcal{C}$ is closed under elementary equivalence and under ultraproducts.
\end{fact}
As a simple consequence of {\L}o\'{s}'s theorem, one obtains the following.
\begin{fact}\label{n2}
The class of Krasner factor hyperfields is closed under ultraproducts.
\end{fact}
Therefore, we will aim to show that the class of Krasner hyperfields is not closed under elementary equivalence. We see below that we can restrict ourselves to the class of Krasner CC-hyperfields.
\begin{lemma}\label{red1}
If the class of Krasner hyperfields is elementary, then the class of Krasner CC-hyperfields is elementary.
\end{lemma}
\begin{proof}
It is obvious, since the condition $(\forall x)(x+x=\{0\})$ is clearly definable in the language from Definition \ref{lang}.
\end{proof}
\begin{remark}\label{d2}
All the assumptions from Theorem \ref{theorem} can be expressed as first-order sentences in the language of hyperfields introduced above (using the explicit definition of the associated projective geometry from Proposition \ref{definable}). Hence, we obtain that the class of Krasner $CC$-hyperfields of dimension at least 2 \emph{is} elementary.
\end{remark}
Because of Remark \ref{d2}, we need to focus on one-dimensional Krasner CC-hyperfields. For convenience, we give names to the following two classes.
\begin{notation}\label{knot}
\begin{enumerate}
  \item Let $\mathcal{K}$ denote the class of Krasner CC-hyperfields of dimension one.
  \item Let $\mathcal{K}^\times$ denote the class of groups which are of the form $L^\times/K^\times$, where $K\subseteq L$ is a field extension of degree 2.
\end{enumerate}
\end{notation}
The next observation explains why the class of groups from Notation \ref{knot}(2) is important for us.
\begin{fact}\label{opposite}
Let $\cH\in \mathcal{K}$. Then we have the following.
\begin{enumerate}
    \item $\cH$ is isomorphic to $L/K^\times$ where $L$ is a field, $K$ is its subfield and $[L:K]=2$.
    \item The hyperaddition in $\cH$ is given by the following formula:
    \begin{align*}
 &x+y=\cH\setminus\{0\}\mbox{ for x $\neq$ y $\neq$ 0,}\\
 &x+x=\{x,0\},\\
 &x+0=0+x=\{x\},
\end{align*}
so it is definable in the language $\{\underline{0}\}$ (just one constant symbol).
\end{enumerate}
\end{fact}
\begin{proof}
Item (1) follows from Remark \ref{desa}, since $\dim(\cP_\cH)=1$ if and only if $\cH$ comes from a field extension of degree 2.

Item (2) follows again from Remark \ref{desa} (and Proposition \ref{definable}), since $\dim(\cP_\cH)=1$ implies that there is only one line in the projective geometry $\cP_\cH$ and this line is the whole space.
\end{proof}

We need one more result from \cite{cc} saying that some hyperfields are definable just from their multiplicative structure. Such a phenomenon is impossible for fields and this is the base of our proof of the main result of this paper (Corollary \ref{main}).
\begin{fact}[Proposition 3.6 in \cite{cc}]\label{opp1}
Let $(G,\cdot)$ be a commutative group (written multiplicatively) of order at least 4. We define $\cH_G$ as $G\cup \{0\}$, where $0$ is a new symbol and extend the group operation on $G$ to the commutative monoid structure on $\cH_G$ by setting $x\cdot 0=0=0\cdot x$. If we define the hyperaddition $+$ on $\cH_G$ as in Fact \ref{opposite}(2), then $(\cH_G,+,\cdot)\in \mathcal{K}$ (see Notation \ref{knot}(1)).
\end{fact}
As a consequence, we directly obtain the following.
\begin{lemma}\label{oppmiddle}
If the class of Krasner CC-hyperfields is elementary, then the class $\mathcal{K}$ is elementary.
\end{lemma}
\begin{proof}
By Facts \ref{opposite} and \ref{opp1}, if $\cH$ is a Krasner CC-hyperfields, then $\cH\in \mathcal{K}$ if and only if the hyperaddition in $\mathcal{K}$ is given by the formula as in Fact \ref{opposite}(2), which is clearly a first-order condition.
\end{proof}
The following easy results outline the further connections between the classes $\mathcal{K}$ and $\mathcal{K}^\times$.
\begin{lemma}\label{opp2}
Let $G$ and $H$ be commutative groups and $\cH_G,\cH_H$ be the corresponding hyperfields as in Fact \ref{opp1}. If $G$ and $H$ are elementarily equivalent (as groups), then $\cH_G$ and $\cH_H$ are elementarily equivalent (as hyperfields).
\end{lemma}
\begin{proof}
Assume that $(G,\cdot)\equiv (H,\cdot)$. By the uniform definition of the monoid operation in $\cH_G,\cH_H$, we get that $(\cH_G,\cdot)\equiv (\cH_H,\cdot)$. Since the hyperaddition in $\cH_G,\cH_H$ is defined by the same formula in the monoid language (we actually only need the extra constant as in Fact \ref{opposite}), we get that $(\cH_G,+,\cdot)\equiv (\cH_H,+,\cdot)$.
\end{proof}
\begin{lemma}\label{opp3}
If the class $\mathcal{K}$ is elementary, then the class $\mathcal{K^\times}$ is closed under elementary equivalence.
\end{lemma}
\begin{proof}
Assume that the class $\mathcal{K}$ is elementary. Let us take $G\in \mathcal{K^\times}$, so there is a hyperfield $\cH\in \mathcal{K}$ such that $G$ is the multiplicative group of $\cH$. By Facts \ref{opposite} and \ref{opp1}, we get that $\cH=\cH_G$ (as hyperfields). We also take a group $H$ such that $G\equiv H$. By Lemma \ref{opp2}, we get that $\cH_H\equiv \cH_G=\cH$. Since the class $\mathcal{K}$ is elementary, we obtain that $\cH_H\in \mathcal{K}$. Therefore, $H\in \mathcal{K}^\times$, which we needed to show.
\end{proof}

With those results at hand, it is enough to show that the class $\mathcal{K}^\times$ is not closed under elementary equivalence, which we will do in Section 3. In the next subsection, we will describe the necessary algebraic tools to obtain this result.

\subsection{Chebotarev's density theorem}\label{seccheb}

We will use one classical result which is a consequence of Chebotarev's density theorem. We need some terminology first (see \cite[Chapter 6]{FrJa}).

Assume that $K\subseteq L$ is a finite Galois extension, where $K$ is either a number field or the field of rational functions $\Ff_q(X)$ over a finite field. Let $\mathcal{O}_K,\mathcal{O}_L$ be the integral closures of $\Zz$ or $\Ff_q[X]$ in $K$ and $L$ respectively. Then $\mathcal{O}_K$ and $\mathcal{O}_L$ are Dedekind rings, hence each non-zero ideal decomposes uniquely into a finite product of maximal ideals. It is also well-known that if $P$ is a maximal ideal of $\mathcal{O}_K$, then we have $P\mathcal{O}_L=Q_1\cdot\ldots \cdot Q_k$, where $Q_i$ are maximal ideals of $\mathcal{O}_L$ and $k\leqslant [L:K]$.
If $k>1$ and the ideals $Q_1,\ldots,Q_k$ are pairwise distinct, then we say that $P$ \emph{splits} in $L$. If $P$ splits in $L$, and if $k=[L:K]$, then we say that $P$ \emph{splits completely} in $L$. In the case of a degree two extension, which is the situation corresponding to the class $\mathcal{K}$ from Notation \ref{knot}(1), the notions of splitting and splitting completely coincide.

We need the following consequence of Chebotarev's density theorem.
\begin{theorem}\label{cheb}
Let $K\subseteq L$ be a finite Galois extension as above. Then there are infinitely many maximal ideals in $\mathcal{O}_K$ which split completely in $L$.
\end{theorem}
\begin{proof}
This directly follows from \cite[Exercise 5(a), page 129]{FrJa} together with \cite[Theorem 6.3.1]{FrJa} (Chebotarev's density theorem) and its interpretation stated in the paragraph before it (the sets of density $0$ are finite).
\end{proof}

\section{Rational rank}\label{mainsec}
In this section, we prove the main results of this paper (Theorem \ref{mainalg} and Corollary \ref{main}).

We need the following notion (see \cite[Section 3.4]{epval}).
\begin{definition}
    The \emph{rational rank} of a commutative group $A$ is the cardinality of a maximal $\Zz$-linearly independent subset of $A$. Following \cite{epval}, we denote it by $\rk(A)$.
\end{definition}
\begin{remark}\label{remrk}
Let  $A$ be a commutative group.
\begin{enumerate}
  \item It is easy to see that we have (see \cite[Section 3.4]{epval}):
  $$\rk(A)=\dim_{\Qq}\left(A\otimes_{\Zz}\Qq\right).$$

  \item If $A_0\leqslant A$, then we have (see \cite[Section 3.4]{epval}):
  $$\rk(A)=\rk(A_0)+\rk\left(A/A_0\right).$$

  \item Other names as ``rank'' or ``Pr\"{u}fer rank'' or ``torsion-free rank'' are sometimes used in this context as well.
\end{enumerate}
\end{remark}
Our main algebraic result is the following (see Notation \ref{knot}).
\begin{theorem}\label{mainalg}
The rational rank of any $A\in\mathcal{K}^\times$ is either $0$ or infinite.
\end{theorem}
\begin{remark}
This result will be generalized in Theorem \ref{gen} to a much wider class of commutative groups coming from multiplicative groups of fields.
\end{remark}
We proceed to the proof of Theorem \ref{mainalg}.
We start with a warm-up example which will be also used in the main proof.
\begin{example}\label{warmup}
We will show here two particular cases of Theorem \ref{mainalg}. For the proof of the first one, we will use a well-known algebraic fact that is also a direct consequence of Chebotarev's theorem. In the proof of the second one, Chebotarev's theorem will not be needed.
\begin{enumerate}
  \item We have
  $$\rk(\Qq[i]^\times/\Qq^\times)=\aleph_0.$$
\begin{proof}
Let us recall that a prime number $p\in \Zz$ splits (equivalently in this case: splits completely) in $\Zz[i]$ if and only if  $p\equiv 1 (\mbox{mod } 4)$ and that there are infinitely many such primes (this is a very special case of Theorem \ref{cheb}). Let us take an infinite sequence $p_1,p_2,\ldots$ of prime numbers which split in $\Zz[i]$.

We have $p_i=r_i\overline{r_i}$, where $r_i$ is a prime element of $\Zz[i]$ and $\overline{r_i}$ is the complex conjugate of $r_i$. Then $r_1,\overline{r_1},r_2,\overline{r_2},\ldots$ is a sequence of pairwise non-associated prime elements of $\Zz[i]$. We will show that the cosets $r_1\Qq^\times,r_2\Qq^\times,\ldots$ are $\Zz$-independent in $\Qq[i]^\times/\Qq^\times$.

Assume not, so there is a non-zero tuple $(n_1,n_2,...,n_k)\in \Zz^k$ such that $r_1^{n_1}r_2^{n_2}\ldots r_k^{n_k}\in\Qq^\times$ (witnessing that $r_1\Qq^\times,r_2\Qq^\times,..,r_k\Qq^\times$ are not $\Zz$-independent in $\Qq[i]^\times/\Qq^\times$). We have:
\[
r_1^{n_1}r_2^{n_2}...r_k^{n_k}=\overline{r_1}^{n_1}\overline{r_2}^{n_2}...\overline{r_k}^{n_k},\]
which contradicts the unique factorization in $\Zz[i]$ (after rearranging the displayed equation in such a way that all the exponents are positive).
\end{proof}

\item We have
  $$\rk(\Qq(X)^\times/\Qq(X^2)^\times)=\aleph_0.$$
\begin{proof}
The proof here is analogous to the proof in item $(1)$. We take pairwise distinct rational numbers $a_1,a_2,\ldots$ and substitute:
$$\Zz\rightsquigarrow \Qq[X^2],\ \ \Zz[i]\rightsquigarrow \Qq[X],\ \  p_i\rightsquigarrow X^2-a_i^2,\ \ r_i\rightsquigarrow X-a_i,$$
where the map $\Qq(X)\ni f\mapsto f(-X)\in \Qq(X)$ plays the role of the complex conjugation.
\end{proof}
\end{enumerate}
\end{example}
We will prove now several technical results which will be used in the proof of Theorem \ref{mainalg}.
\begin{lemma}\label{rank0}
Let $F\subseteq K$ be a field extension of degree $2$. Suppose that $F$ is algebraic over a finite field or the extension $F\subseteq K$ is not Galois. Then we have:
$$\rk\left(K^\times/F^\times\right)=0.$$
\end{lemma}
\begin{proof}
If $F$ is algebraic over a finite field, then $K^\times$ is torsion. Therefore, $K^\times/F^\times$ is torsion as well and $\rk\left(K^\times/F^\times\right)=0.$

Assume that the extension $F\subseteq K$ is not Galois. Then $\cha(F)=2$ and the extension $F\subseteq K$ is purely inseparable (since $[K:F]=2$). Therefore, $K^2\subseteq F$ (using again that $[K:F]=2$), hence the group $K^\times/F^\times$ is torsion again.
\end{proof}

For the proof of Theorem \ref{mainalg}, it is enough now to show the following:

\begin{proposition}\label{rankinfty}
Let $F\subseteq K$ be a Galois field extension of degree $2$ such that $F$ is not algebraic over a finite field. Then the rational rank of
$K^\times/F^\times$ is infinite.
\end{proposition}
The next result allows as to reduce the proof of Proposition \ref{rankinfty} to several special cases.
\begin{lemma}\label{reduce}
Let $K\subseteq L$ be a field extension satisfying the assumption of Proposition \ref{rankinfty}. Then either
\begin{enumerate}
    \item $\Qq(X)^\times/\Qq(X^2)^\times$, or
    \item $F^\times/E^\times$, where $E\subseteq F$ is a number field extension of degree 2, or
    \item $\Ff_{p}(X)^\times/\Ff_{p}(X^2)^\times$ for $p\neq 2$, or
    \item $\Ff_{2}(X)^\times/\Ff_{2}(X^2+X)^\times$,
\end{enumerate}
embeds into $L^\times/K^\times$.
\end{lemma}
\begin{proof}
We consider several cases.
\\
\\
{\bf Case I} $\cha(K)=0$.
\\
\\
By a standard argument (which works whenever $\cha(K)\neq 2$), there is $\alpha\in L$ such that $\alpha^2\in K$ and $L=K(\alpha)$, so $K\subseteq L$ is a Kummer extension. Since we are in the characteristic 0 case, we can assume that $\Qq$ is a subfield of $K$. By our choice of $\alpha$, we have that $\Qq(\alpha)\cap K=\Qq(\alpha^2)$. Therefore, $\Qq(\alpha)^\times/\Qq(\alpha^2)^\times$ embeds into $L^\times/K^\times$. If $\alpha$ is transcendental over $\Qq$, then we are in the item $(1)$ situation. Otherwise, we are in the item $(2)$ situation.
\\
\\
{\bf Case II} $\cha(K)=p>2$.
\\
\\
As in Case I, there is $\alpha\in L$ such that $\alpha^2\in K$ and $L=K(\alpha)$, , so $K\subseteq L$ is a Kummer extension again. We can assume that $\Ff_p$ is a subfield of $K$. By our assumptions, there is $t\in K$ which is transcendental over $\Ff_p$. If $\alpha$ is algebraic over $\Ff_p$, then we replace $\alpha$ with $t\alpha$. If $\alpha$ is transcendental over $\Ff_p$, then we keep $\alpha$ as it is. After this possible replacement, we moreover obtain that $\alpha$ is transcendental over $\Ff_p$. As in Case I,  $\Ff_p(\alpha)^\times/\Ff_p(\alpha^2)^\times$ embeds into $L^\times/K^\times$ and we are in the item $(3)$ situation.
\\
\\
{\bf Case III} $\cha(K)=2$.
\\
\\
Since the extension $K\subseteq L$ is Galois, by Artin-Schreier theory there is $\alpha\in L$ such that $\alpha^2+\alpha\in K$ and $L=K(\alpha)$. We proceed now as in Case II, however we (possibly) need to replace $\alpha$ with $t+\alpha$ (rather than $t\alpha$). In the end, we are in the item $(4)$ situation.
\end{proof}

We are ready to show the main algebraic result of this section.

\begin{proof}[Proof of Prop. \ref{rankinfty}]
It is clear that if a commutative group $A$ embeds into a commutative group $B$, then we have $\rk(A)\leqslant \rk(B)$. Therefore, it is enough to consider the four cases given by the statement of Lemma \ref{reduce}.
\\
\\
{\bf Case (1)}
\\
$\rk(\Qq(X)^\times/\Qq(X^2)^\times)=\aleph_0$.
\\
\\
This is covered by Example \ref{warmup}(2).
\\
\\
{\bf Case (2)}
\\
$\rk(L^\times/K^\times)=\aleph_0$, where $K\subseteq L$ is a number field extension of degree 2.
\\
\\
This is a generalization of Example \ref{warmup}(1). Let  $\cO_L$ and $\cO_K$ denote the rings of algebraic integers of $L$ and $K$ respectively (as in Section \ref{seccheb}) and let $\gal(L/K)=\{\id,\sigma\}$.

By Theorem \ref{cheb}, there is an infinite sequence $P_1,P_2,\ldots$ of prime ideals in $\cO_K$ which split completely in $L$
(as we remarked earlier, in the case of the degree 2 extension, splitting implies splitting completely). Then for each $i$, we have:
$$P_i=Q_i\sigma(Q_i),$$
where $Q_i$ is a maximal ideal of $\cO_L$ and all the ideals $Q_1,\sigma(Q_1),Q_2,\sigma(Q_2),\ldots$ are pairwise distinct.

Using the Prime Avoidance Lemma (see e.g. \cite[Section 3.2]{comm}), we choose $a_1,a_2,\ldots \in \cO_L$ such that
\begin{itemize}
    \item $a_1\in Q_1\setminus\sigma(Q_1)$,
    \item $a_2\in Q_2\setminus(Q_1\cup \sigma(Q_1)\cup\sigma(Q_2))$,\\
    and in general:
    \item $a_i\in Q_i\setminus(Q_1\cup \ldots \cup Q_{i-1}\cup \sigma(Q_1)\cup\ldots \cup\sigma(Q_i))$.
\end{itemize}
Similarly as in Example \ref{warmup}(1) (although without using the unique decomposition), we will show that the cosets $a_1K^\times,a_2K^\times,\ldots$ are $\Zz$-independent in $L^\times/K^\times$.

Assume not, so there is a nonzero tuple $(n_1,\ldots,n_k)\in \Zz^k$ such that $a_1^{n_1}\ldots a_k^{n_k}\in K^\times$ (witnessing that $a_1K^\times,\ldots,a_kK^\times$ are not $\Zz$-independent in $L^\times/K^\times$). Since the tuple $(n_1,\ldots,n_k)$ is nonzero, there is a smallest $r\leqslant k$, such that $n_r\neq 0$. Without loss of generality, we can assume that $n_r>0$.

Since $a_{r}^{n_r}\ldots a_k^{n_k}\in K^\times$, we get
   \begin{equation}
a_{r}^{n_r}\ldots a_k^{n_k}=\sigma(a_{r}^{n_r}\ldots a_k^{n_k})=\sigma(a_{r})^{n_r}\ldots \sigma(a_k)^{n_k}.
\tag{$*$}
\end{equation}
By our choice of the sequence $a_1,a_2,\ldots$, we have:
   \begin{equation}
a_r\in Q_r,\ \ a_{r+1},\ldots,a_{_k}\notin Q_r,  \ \ \sigma(a_{r}),\ldots,\sigma(a_{k})\notin Q_r.\tag{$**$}
\end{equation}
Using $(*),(**)$ and doing a similar rearrangement as in Example \ref{warmup}(1), we get that there are elements $b\in \cO_L, c_1,\ldots,c_m\in \cO_L\setminus Q_r$ such that
   \begin{equation}
a_{r}^{n_r}b=c_1\ldots c_m. \tag{$***$}
\end{equation}
Since $n_r>0$, $a_r\in Q_r$ and $Q_r$ is prime, the equality $(***)$ yields a contradiction, since its left-hand side belongs to $Q_r$ and its right-hand side does not.
\\
\\
{\bf Case (3)}
\\
$\rk(\Ff_{p}(X)^\times/\Ff_{p}(X^2)^\times)=\aleph_0$, where $p\neq 2$.
\\
\\
The proof is almost the same as in Example \ref{warmup}(2), however now we substitute:
$$\Zz\rightsquigarrow \Ff_p[X^2],\ \ \Zz[i]\rightsquigarrow \Ff_p[X].$$
The choice of irreducible polynomials in $\Ff_p[X^2]$ which split in $\Ff_p[X]$ is not as straightforward as in Example \ref{warmup}(2). To obtain such a sequence, we again use Theorem \ref{cheb}, which also covers Case (3). With such a sequence at hand, we continue the proof as in Example \ref{warmup}(2).
\\
\\
{\bf Case (4)}
\\
$\rk(\Ff_{2}(X)^\times/\Ff_{2}(X^2+X)^\times)=\aleph_0$.
\\
\\
It is almost identical to the Case (3) situation, we just use the ring $\Ff_2[X^2+X]$ instead of the ring $\Ff_p[X^2]$.
\end{proof}
This concludes the proof of Proposition \ref{rankinfty}, and hence also the proof of Theorem \ref{mainalg}.\\

Our main model-theoretic result is below.
\begin{corollary}\label{main}
The class of Krasner hyperfields is not elementary.
\end{corollary}
\begin{proof}
If the class of Krasner hyperfields is elementary, then the class $\mathcal{K^\times}$ of groups (see Notation \ref{knot}(2)) is closed under elementary equivalence by Lemmas \ref{red1}, \ref{oppmiddle}, \ref{opp3}. We will show that this is not the case.

Since $\Cc^\times$ is divisible, we have:
\[
\mathcal{K}^\times\ni \Cc^\times/\Rr^\times \cong A\oplus\bigoplus_p C_{p^\infty}\equiv \Qq\oplus\bigoplus_p C_{p^\infty},
\]
where $C_{p^\infty}$ is the Pr\"{u}fer $p$-group and $A$ is a vector space over $\Qq$ of dimension continuum. The isomorphism above follows from the classification of divisible commutative groups (see \cite[Theorem 5 in Section 4]{kaplanskyinf}) and  the elementary equivalence follows from the Szmielew's description (see \cite{szmielew}) of the theories of commutative groups (this elementary equivalence can be also directly shown by taking ultraproducts).

However, the rational rank of $\Qq\oplus\bigoplus_p C_{p^\infty}$ is 1, so this group does not belong to $\mathcal{K}^\times$ by Theorem \ref{mainalg}.
\end{proof}

\section{Generalization, conjecture and question}\label{specsec}
In this section we improve on Theorem \ref{mainalg} and we also discuss some model-theoretical problems related with hyperfields and the algebraic methods used in this paper.

We show now the following improvement of Theorem \ref{mainalg} to its proper generality.
\begin{theorem}\label{gen}
Let $F\subseteq K$ be an arbitrary field extension. Then both $\rk(F^\times)$ and $\rk(K^\times/F^\times)$ are $0$ or infinite.
\end{theorem}
\begin{proof}[Sketch of Proof]
If $F$ is algebraic over a finite field, then $F^\times$ is torsion, so $\rk(F^\times)=0$. If $F$ is not algebraic over a finite field, then either $\Qq$ embeds into $F$ or $\Ff_p(X)$ embeds into $F$ for some prime number $p$. So, for the case of $\rk(F^\times)$, it is enough to notice that
$$\rk(\Qq^\times)=\rk(\Ff_p(X)^\times)=\aleph_0,$$
which follows (as in Example \ref{warmup}) from the fact that both of these fields are fraction fields of unique factorization domains with infinitely many pairwise non-associated prime elements.

We move now to the case of a field extension $F\subseteq K$. If the extension $F\subseteq K$ is not algebraic, we take a transcendental $t\in K$ and then $\rk(F(t)^\times/F^\times)$ is infinite by a similar argument as above.

If the extension $F\subseteq K$ is purely inseparable or $K$ is contained in the algebraic closure of a finite field, then $\rk(K^\times/F^\times)=0$ as in the proof of Lemma \ref{rank0}.

Therefore, we can assume that $F\subseteq K$ is a finite extension which is not contained in the algebraic closure of a finite field and which is also not purely inseparable. We aim to show that $\rk(K^\times/F^\times)$ is infinite. Let us take the field tower $F\subseteq K_0\subseteq K$, where the first extension is separable and non-trivial and the second one is purely inseparable. We have the following exact sequence:
$$1\to K_0^\times/F^\times\to K^\times/F^\times\to K^\times/K_0^\times\to 1$$
and we know that $\rk(K^\times/K_0^\times)=0$, so using Remark \ref{remrk}(2), we obtain that $\rk(K_0^\times/F^\times)=\rk(K^\times/F^\times)$, so we can moreover assume that $F\subseteq K$ is finite, separable, and $F=\Ff_p(X)$ or $F$ is a number field. Let $F\subseteq L$ be the normal closure of $F\subseteq K$ and $n:=[K:F]>1$. By Theorem \ref{cheb}, there are infinitely many prime ideals $P_1,P_2,\ldots$ of $\mathcal{O}_F$ which split completely in $\mathcal{O}_L$, so for each $i$, we also have $P_i\mathcal{O}_K=Q_{i1}\ldots Q_{in}$, where $Q_{ij}$'s are maximal ideals in $\mathcal{O}_K$. We take $a_1,a_2,\ldots \in \mathcal{O}_K$ such that for each $i$, we have:
$$a_i\in Q_{i1}\setminus \left(\bigcup_{j=1}^{i-1}Q_{j1}\cup \bigcup_{j=1}^i\bigcup_{k=2}^lQ_{jk}\right).$$
Then the proof of the Case $(2)$ situation from Lemma \ref{reduce} works after taking $\sigma\in \gal(L/F)$ such that $\sigma(a_r)\notin Q_{r1}\mathcal{O}_L$.
\end{proof}
Therefore (as in the proof of Corollary \ref{main}), the following class of groups:
$$\{K^\times\ |\ \text{$K$ is a field}\}$$
is \emph{not} elementary. Interestingly, a similar phenomenon appeared in \cite{HKTY} where the authors consider model completeness of groups of rational points of algebraic groups. One can ask the following.
\begin{question}\label{q}
Let $\mathbb{G}$ be a group scheme over $\Zz$. Are the following two conditions on $\mathbb{G}$ equivalent?
\begin{enumerate}
  \item The class
  $$\{\mathbb{G}(K)\ |\ \text{$K$ is a field}\}$$
  is elementary.

  \item If $K$ is a model complete field, then $\mathbb{G}(K)$ is a model complete group.
\end{enumerate}
\end{question}
The multiplicative group scheme $\gm$ fails both items $(1)$ and $(2)$ above. On the other hand, semisimple or unipotent algebraic groups seem to satisfy both these items, which is work in progress related to \cite{HKTY}. Therefore, we do not have counterexamples to the equivalence in Question \ref{q}. Actually, if item $(1)$ holds, then (as in \cite{HKTY}) it is usually an important step for proving that item $(2)$ holds. The fact that item $(1)$ holds for certain simple algebraic groups follows from \cite{segtent} and \cite{thomasthesis}.
\\
\\
While trying to understand hyperfields (or any other structures) model-theoretically, it is natural to ask first what are the ``model-theoretically simplest'', that is \emph{strongly minimal}, hyperfields. We propose the following.
\begin{conjecture}\label{c}
A hyperfield is strongly minimal if and only if it is either a strongly minimal field (i.e., an algebraically closed field) or a hyperfield where the hyperaddition is definable in the structure of its multiplicative group which is strongly minimal.
\end{conjecture}
Since any infinite commutative group can be expanded to a hyperfield where the hyperaddition is definable just from one constant symbol (see Fact \ref{opp1}) there are plenty hyperfields as after ``or'' in the conjecture above.
\bibliographystyle{plain}
\bibliography{harvard}

\end{document}